\documentclass{amsart}
\usepackage{latexsym,amsmath,amssymb}
\newtheorem{thm}{Theorem}[section]
\newtheorem{cor}[thm]{Corollary}
\newtheorem{exam}[thm]{Example}
\newtheorem{lem}[thm]{Lemma}
\newtheorem{prop}[thm]{Proposition}
\theoremstyle{definition}\newtheorem{defn}[thm]{Definition}
\theoremstyle{remark}
\newtheorem{rem}[thm]{Remark}
\numberwithin{equation}{section}

\begin{document}

\title[]
{Moore-Penrose inverse and applications to linear equations }

\author{\sc\bf Y. Estaremi and S. Shamsi }
\address{\sc Y. Estaremi and S. Shamsi}
\email{y.estaremi@gu.ac.ir}
\email{s.shamsi@gmail.com}

\address{Department of Mathematics and Computer Sciences,
Golestan University, Gorgan, Iran,}
\address{Department of Mathematics, Noor University, p. o. box: 19395-3697, Tehran, Iran.}

\thanks{}

\thanks{}



\date{}

\dedicatory{}

\commby{}

\begin{abstract}
In the present paper we investigate Moore-Penrose inverse and characteristic matrix of unbounded WCT operators on the Hilbert space $L^2(\mu)$. Also, we obtain some applications of the Moore-Penrose inverse of unbounded operators on the Hilbert space $\mathcal{H}$ to variational regularization problem. Moreover,some examples are provided to illustrate the applications of our results in linear equations and specially Fredholm integral equations.
\end{abstract}

\maketitle

\section{ \sc\bf Introduction and Preliminaries}
The theory of conditional expectation operators is a connection between measure theory and operator theory. Indeed for any $\sigma$-subalgebra of a $\sigma$-algebra in a measure space we can define a unique bounded operator on the space of measurable function spaces, named conditional expectation with respect to the chosen $\sigma$-subalgebra. Also, the class of composition of conditional expectation and multiplication operators contains a large class of linear operators(bounded and unbounded) like integral operator and weighted composition operators. The theory of weighted conditional type(WCT) operators have been studied by, for example Moy, \cite{mo}, De Pagter and Grobler \cite{gd}, Rao, \cite{rao}, Douglas, \cite{dou}, Lambert, \cite{lam} and Herron, \cite{her}. Specially, bounded WCT operators have been studied by many authors on $L^p$-spaces. Also, we initiated to study unbounded WCT operators on $L^2(\mu)$-spaces in \cite{yc} and on $L^p(\mu)$-spaces in \cite{ye1, xy}. Most of the results in the literature are almost completely pure. In this paper we have obtained that the Moore-Penrose inverse of WCT operators ($T$) at each single point $f\in L^2(\mu)$, is the minimal-norm solution of the linear equation $Tu=f$. Also, it is true for any unbounded densely defined operator $T$ on the Hilbert space $\mathcal{H}$.\\



Let $(X,\Sigma,\mu)$ be a $\sigma$-finite measure space. For any $\sigma$-finite subalgebra
$\mathcal{A}\subseteq
 \Sigma$, the $L^2$-space,
$L^2(X,\mathcal{A},\mu_{|_{\mathcal{A}}})$ is abbreviated  by
$L^2(\mathcal{A})$, and its norm is denoted by $\|.\|_2$. All
comparisons between two functions or two sets are to be
interpreted as holding up to a $\mu$-null set. The support of a
measurable function $f$ is defined as $S(f)=\{x\in X; f(x)\neq
0\}$. We denote the vector space of all (equivalence classes of)
almost everywhere finite valued measurable functions on $X$ by
$L^0(\Sigma)$.\\

\vspace*{0.3cm} For a $\sigma$-finite subalgebra
$\mathcal{A}\subseteq\Sigma$, the conditional expectation operator
associated with $\mathcal{A}$ is the mapping $f\rightarrow
E^{\mathcal{A}}f$, defined for all non-negative, measurable
function $f$ as well as for all $f\in L^2(\Sigma)$, where
$E^{\mathcal{A}}f$, by the Radon-Nikodym theorem, is the unique
$\mathcal{A}$-measurable function satisfying
$$\int_{A}fd\mu=\int_{A}E^{\mathcal{A}}fd\mu, \ \ \ (A\in \mathcal{A}).$$
As an operator on $L^{2}({\Sigma})$, $E^{\mathcal{A}}$ is
idempotent and $E^{\mathcal{A}}(L^2(\Sigma))=L^2(\mathcal{A})$. If
there is no possibility of confusion, we write $E(f)$ in place of
$E^{\mathcal{A}}(f)$. A detailed discussion of the properties
of conditional expectation may be found in \cite{rao}. Let $u\in L^0(\Sigma)$. Then $u$ is said to be
conditionable with respect to $E$ if $u\in\mathcal{D}(E):=\{f\in
L^0(\Sigma): E(|f|)\in L^0(\mathcal{A})\}$.  Throughout this paper we take $u$
and $w$ in $\mathcal{D}(E)$. We now define the class of operator under investigation.

\begin{defn}
Let $(X,\Sigma,\mu)$ be a $\sigma$-finite measure space and let $\mathcal{A}$ be a
$\sigma$-subalgebra of $\Sigma$, such that $(X,\mathcal{A},\mathcal{A})$ is also $\sigma$-finite. Suppose that $E$ is the corresponding conditional
expectation operator on $L^2(\Sigma)$ relative to $\mathcal{A}$. If $w,u \in L^0(\Sigma)$ such that $uf$ is conditionable and $wE(uf)\in L^2(\Sigma)$, for all $f\in L^2(\Sigma)$, then the corresponding weighted conditional type(WCT) operator is the linear
transformation $M_wEM_u:L^2(\Sigma)\rightarrow L^2(\Sigma)$ defined by $f\rightarrow wE(uf)$.\\
\end{defn}
In this paper we are concerned to unbounded WCT operators on the Hilbert space $L^2(\Sigma)$ and we will discuss Moore-Penrose inverse of these operators. Also, we give a spectral characterization of closed range WCT operators. Moreover, we find the characteristic matrix of WCT operators. Finally, we obtain that the value of any unbounded densely defined operator $T$ on the Hilbert space $\mathcal{H}$ at any single point $f\in \mathcal{H}$, is the minimal-norm solution of the linear equation $Tu=f$.

\section{\sc\bf Moore-Penrose inverse of unbounded WCT operators}
Let $u,w\in \mathcal{D}(E)$ and $T=M_wEM_u$ be the WCT operator on the Hilbert space $L^2(\mu)$. Here we need to recall conditional type Holder inequality on $L^2(\mu)$, it asserts that for every $f,g \in L^2(\mu)$, the following inequality holds:

$$| E(fg)|^2 \leq E(| f|^2)E(| g|^2).$$
For $f\in \mathcal{D}(M_wEM_u)\subseteq L^2(\mu)$, we have $Tf=M_wEM_u(f)=wE(uf)$. We know that $S(| w|)\subseteq S(E(| w|^2))$ and also by conditional type Holder inequality we have

$$| wE(fg)|^2 \leq | w|^2 E(| f|^2)E(| g|^2).$$
Since $S(| Tf|)\subseteq S(E(| Tf|^2)$, then we have the following lemma.
\begin{lem}\label{l2.1}
Let $S=S(E(| w|^2))$ and $G=S(E(| u|^2))$ and $T=M_wEM_u$ be the WCT operator on $L^2(\mu)$. Then $S(Tf)\subseteq S\cap G$.
\end{lem}
Let $T=M_wEM_u$ be the WCT operator on $L^2(\mu)$, i.e.,
 $$T=M_wEM_u:\mathcal{D}(T)\subseteq L^2(\mu)\rightarrow \mathcal{R}(T)\subseteq L^2(\mu).$$
 As we have studied unbounded WCT operators in \cite{yc}. The WCT operator $T$ is densely defined on $L^2(\mu)$ if and only if $E(| w|^2)E(| u|^2)<\infty$, a.e., $\mu$. This gives us that the WCT operator $M_wEM_u$ is a densely defined operator on $L^2(\mu)$ if and only if the WCT operator $M_uEM_w$ is a densely defined operator on $L^2(\mu)$. Also, for densely defined operator $T=M_wEM_u$ on $L^2(\mu)$, we have $T^*=M_{\bar{u}}EM_{\bar{w}}$. So by the above observations we get that $T=M_wEM_u$ is densely defined if and only if $T^*$ is densely defined.
  This implies that if $T=M_wEM_u$ is a densely defined operator on $L^2(\mu)$, then it is closed. And consequently for densely defined operator $T=M_wEM_u$, we have $T^{**}=T$. Therefore, if $T=M_wEM_u$ and $E(| w|^2)E(| u|^2)< \infty$, \ \ a.e., $\mu$, then $T$ and $T^*$ both are closed. Now we provide some technical lemmas for later use.

  \begin{lem}\label{l2.2}
  Let $EM_u$ be a linear operator on $L^2(\mu)$ such that $EM_u(f)=E(uf)$ for $f\in \mathcal{D}(EM_u)$. Then $L^2(\mathcal{A})\subseteq\mathcal{D}(EM_u)$ if and only if $\int_A | E(u)|^2 d\mu<\infty$, for all $A\in \mathcal{A}$, with $0<\mu(A)<\infty$.
  \end{lem}
  \begin{proof}
  $L^2(\mathcal{A})\subseteq \mathcal{D}(EM_u)$ if and only if $\int_X| E(u)f|^2d\mu=\int_X|E(uf)|^2d\mu<\infty$, for every $f\in L^2(\mathcal{A})$. Since $L^2(\mathcal{A})$ is $\sigma$-finite, then simple functions are dense in $L^2(\mathcal{A})$. Therefore we have $\int_X|E(u)f|^2d\mu<\infty$, for all $f\in L^2(\mathcal{A})$ if and only if $\int_A |E(u)|^2d\mu<\infty$, for all $A\in \mathcal{A}$ with $\mu(A)<\infty$.
  \end{proof}

\begin{lem}\label{l2.3}
If the measure space $(X,\Sigma,\mu)$ is $\sigma$-finite and $f$ is a $\Sigma$-measurable function such that $f<\infty$, a.e., $\mu$, then there exists an increasing sequence of $\Sigma$-measurable sets $\{B_n\}_{n\in \mathbb{N}}$ such that $f.\chi_{B_n}<n$, $\mu(B_n)<\infty$ and $X=\cup^{\infty}_{n=1} B_n$.
\end{lem}
\begin{proof}
It is an easy exercise.
\end{proof}

\paragraph{$(\bigstar)$ Let $(X, \Sigma, \mu)$ be a measure space such that if $X=\cup_{n\in \mathbb{N}}X_n$ and $\mu(X_n)<\infty$, for all $n$, in which $\{X_n\}$ is an increasing sequence, then for each $A\in \Sigma$ with $\mu(A)<\infty$, there exists $n\in \mathbb{N}$ such that $A\subseteq X_n$. }

\begin{prop}\label{p2.4}
Let $(X,\Sigma, \mu)$ be a $\sigma$-finite measure space that satisfies the condition $(\bigstar)$. If the WCT operator $M_wEM_u$ is densely defined on $L^2(\mu)$, then $L^2(\mathcal{A})\subseteq \mathcal{D}(EM_{wu})$.
\end{prop}
\begin{proof}
Let $M_wEM_u$ be densely defined on $L^2(\mu)$. Then we have $E(| u|^2)E(| w|^2)<\infty$, a.e., $\mu$. Hence by conditional type Holder inequality we get that $| E(uw)|^2<\infty$, a.e., $\mu$. Now by Lemma \ref{l2.3} we get that there is an increasing sequence $\{B_n\}_{n\in \mathbb{N}}\subseteq\mathcal{A}$ such that $| E(uw)|^2\chi_{B_n}<n$, $\mu(B_n)<\infty$ and $X=\cup^{\infty}_{n=1}B_n$. Let $A\in \mathcal{A}$ such that $\mu(A)<\infty$. Then by $(\bigstar)$, there exists $n\in \mathbb{N}$ for which $A\subseteq B_n$. This implies that

$$\int_A | E(uw)|^2d\mu\leq\int_{B_n}| E(uw)|^2d\mu\leq n\mu(B_n)<\infty.$$
Thus by Lemma \ref{l2.2} we get that $L^2(\mathcal{A})\subseteq \mathcal{D}(EM_{uw})$.
\end{proof}
Now we have the following Remark.
\begin{rem}\label{r2.5}
If the WCT operator $M_wEM_u$ is densely defined on $L^2(\mu)$ such that the underlying measure space satisfies the condition $(\bigstar)$, then by Corollary 2.11 of \cite{xy} and the Proposition \ref{p2.4} we get that
$$\sigma(M_wEM_u)\cup\{0\}=\text{ess.range}(E(uw))\cup\{0\}.$$
\end{rem}
For the Hilbert spaces $\mathcal{H}_1$ and $\mathcal{H}_2$, the set of all closed linear operators $T$, with $\mathcal{D}(T)\subseteq \mathcal{H}_1$ and $\mathcal{R}(T)\subseteq \mathcal{H}_2$, is denoted by $\mathcal{C}(\mathcal{H}_1,\mathcal{H}_2)$. Also, a linear operator $T\in \mathcal{C}(\mathcal{H}_1,\mathcal{H}_2)$ is called regular if it is densely defined. As we stated in the beginning of this section every densely defined WCT operator is closed, so every densely defined WCT operator is regular.
Here we recall the definition of the Moore-Penrose inverse and some related concepts.
\begin{defn}\label{d2.7}
Let $\mathcal{H}_1$ and $\mathcal{H}_2$ be Hilbert spaces and $T\in \mathcal{C}(\mathcal{H}_1,\mathcal{H}_2)$ be a densely defined operator. Then there exists a unique densely defined operator $T^{\dagger}\in \mathcal{C}(\mathcal{H}_1,\mathcal{H}_2)$, for which $\mathcal{D}(T^{\dagger})=\mathcal{R}(T)\oplus \mathcal{R}(T)^{\perp}$ and has the following properties:\\
1) $TT^{\dagger}y=P_{\overline{\mathcal{R}(T)}}y, \ \ \ \ \ \ \ \ y\in \mathcal{D}(T^{\dagger})$;\\
2) $T^{\dagger}Tx=P_{\mathcal{N}(T)^{\perp}}x, \ \ \ \ \ \ \ \ x\in\mathcal{D}(T)$;\\
3) $\mathcal{N}(T^{\dagger})=\mathcal{R}(T)^{\perp}$.\\
The unique operator $T^{\dagger}$ is called the Moore-Penrose inverse of $T$.
\end{defn}
In \cite{js} the authors have obtained the Moore-Penrose inverse of the bounded WCT operator $T=M_wEM_u:L^2(\mu)\rightarrow L^2(\mu)$, in the special case of $w,u\geq 0$, as follow:

 $$T^{\dagger}f=\frac{\chi_{K}}{E(u^2)E(w^2)}M_uEM_w(f), \ \ \ \ \ \ f\in \mathcal{D}(T^{\dagger}),$$

in which $K=S(E(w))\cap S(E(u))$. But in the case that $w$ and $u$ are complex valued measurable functions for closed range bounded WCT operator $T=M_wEM_u$ on $L^2(\mu)$ we have

$$T^{\dagger}=M_{\frac{\chi_{S}}{E(|u|^2)E(|w|^2)}}M_{\bar{u}}EM_{\bar{w}},$$
in which $S=S(E(|u|^2))\cap S(E(|w|^2))$. Since $\chi_STf=Tf$ and $\chi_ST^*g=T^*g$, for $f,g\in L^2(\mu)$, then easily one can check that $T^{\dagger}$ satisfies the conditions of the definition of Moore-Penrose inverse of closed range bounded operators, i.e.,

$$TT^{\dagger}T=T, \ \ \ T^{\dagger}TT^{\dagger}=T^{\dagger}, \ \ \ (T^{\dagger}T)^*=T^{\dagger}T, \ \ \ (TT^{\dagger})^*=TT^{\dagger}.$$

Moreover, if $T=M_wEM_u$ is an unbounded densely defined operator on $L^2(\mu)$, then it is closed. So by definition \ref{d2.7} there exists a unique closed densely defined operator $T^{\dagger}$ that satisfies with conditions of the Definition \ref{d2.7}. Similar to the bounded case, for unbounded densely defined WCT operator $T=M_wEM_u$, easily we get that

$$T^{\dagger}f=\frac{\chi_S}{E(|u|^2)E(|w|^2)}\bar{u}E(\bar{w}f), \ \ \ \ \ f\in \mathcal{D}(T^{\dagger}).$$
Indeed, $T^{\dagger}=M_{\frac{\chi_S}{E(|u|^2)E(|w|^2)}}T^*$ and so
$\mathcal{N}(T^{\dagger})=\mathcal{N}(T^*)=\mathcal{R}(T)^{\perp}$. Also for $f\in \mathcal{D}(T^{\dagger})$ we have
\begin{align*}
T^{\dagger}T(T^{\dagger}f)&=\frac{\chi_S}{(E(|u|^2))^2(E(|w|^2))^2}M_{\bar{u}}EM_{\bar{w}}M_wEM_uM_{\bar{u}}EM_{\bar{w}}f\\
&=\frac{E(|u|^2)E(|w|^2)\chi_S}{(E(|u|^2))^2(E(|w|^2))^2}M_{\bar{u}}EM_{\bar{w}}f\\
&=\frac{\chi_S}{E(|u|^2)E(|w|^2)}M_{\bar{u}}EM_{\bar{w}}f\\
&=T^{\dagger}f.
\end{align*}
This implies that $T^{\dagger}TT^{\dagger}=T^{\dagger}$ and $(T^{\dagger}T)^2=T^{\dagger}T$. Also, by definition, it is clear that $(T^{\dagger}T)^*=T^{\dagger}T$. Therefore $T^{\dagger}T=P_{\overline{\mathcal{R}(T^{\dagger})}}$. Similarly we get that $TT^{\dagger}=P_{\overline{\mathcal{R}(T)}}$. Therefor we have the following theorem.
\begin{thm}\label{t2.9}
Let $T=M_wEM_u$ be unbounded densely defined WCT operator on $L^2(\mu)$. Then the Moore-Penrose inverse of $T$ is as follow:
$$T^{\dagger}=M_{\frac{\chi_S}{E(|u|^2)E(|w|^2)}}M_{\bar{u}}EM_{\bar{w}}, \text{in which} \ \ S=S(E(|u|^2))\cap S(E(|w|^2)).$$
\end{thm}

In the next Theorem we give an equivalent condition for unbounded WCT operators to have closed range.
\begin{thm}\label{t2.6}
 If WCT operator $M_wEM_u$ is densely defined on $L^2(\mu)$ such that the underlying measure space satisfies with ($\bigstar$) condition, then $\mathcal{R}(M_wEM_u)$ is closed if and only if $E(|u|^2)E(|w|^2)\geq r$, $\mu$, a.e., for some $r>0$.
\end{thm}
\begin{proof} By the Remark \ref{r2.5} we have $\sigma(M_wEM_u)\cup\{0\}=\text{ess.range}(E(|u|^2)E(|w|^2))\cup\{0\}$. Also, for $T=M_wEM_u$ we have $T^*T=M_{\bar{u}E(|w|^2)}EM_u$. Therefore by Theorem 3.3 of \cite{kn} we get that $\mathcal{R}(T)$ is closed if and only if $\sigma(M_{\bar{u}E(|w|^2)}EM_u)\subseteq [r,\infty)\cup\{0\}$, for some $r>0$. This implies that $\mathcal{R}(T)$ is closed if and only if $\text{ess.range}(E(|u|^2)E(|w|^2)\subseteq [r,\infty)\cup \{0\}$, for some $r>0$, if and only if $E(|u|^2)E(|w|^2)\geq r$, $\mu$, a.e., for some $r>0$.
\end{proof}
 Here we recall some results of \cite{bi} as the next proposition, in which $\mathcal{C}(T)=\mathcal{D}(T)\cap\mathcal{N}(T)^{\perp}$ and the reduced minimum modulus of $T$ is defined by $\gamma(T)=\inf\{\|Tx\|:x\in \mathcal{C}(T), \ \ \|x\|=1\}$.
\begin{prop}\label{p2.8}\cite{bi}
Let $\mathcal{H}_1$ and $\mathcal{H}_2$ be Hilbert spaces and $T:\mathcal{H}_1\rightarrow \mathcal{H}_2$ be a closed operator. Then the following are equivalent.\\
\begin{enumerate}
 \item $\mathcal{R}(T)$ is closed.
\item $\mathcal{R}(T^*)$ is closed.
\item $T_0=T|_{\mathcal{C}(T)}$ has a bounded inverse.
\item $\gamma(T)>0$.
\item $T^{\dagger}$ is bounded.
\item $\gamma(T)=\gamma(T^*)$.
\item $\mathcal{R}(T^*T)$ is closed.
\item $\mathcal{R}(TT^*)$ is closed.
\end{enumerate}
\end{prop}
Since $T^*=M_{\bar{u}}EM_{\bar{w}}$, then by Theorem \ref{t2.6} and Proposition \ref{p2.8} we have the following corollary.
\begin{cor}
If the WCT operator $T=M_wEM_u$ is densely define on Hilbert space $L^2(\mu)$, then $\mathcal{R}(T)$ is closed if and only if $\mathcal{R}(T^*)$ is closed if and only if $E(|u|^2)E(|w|^2)\geq r$, $\mu$, a.e., for some $r>0$.
\end{cor}
Let $T\in \mathcal{C}(\mathcal{H}_1, \mathcal{H}_2)$. By definition of reduced minimum modulus of $T$, for each $x\in \mathcal{D}(T)$ with $\|x\|=1$, we can write
$\gamma(T)^2\leq \|Tx\|^2=\langle T^*Tx,x\rangle$ and so $\|T\|^2=\|T^*T\|\geq \gamma(T)^2$. Then we have $\|T\|\geq \gamma(T)$. If $T=M_wEM_u$ and $\mathcal{R}(T)$ is closed, then by Proposition \ref{p2.8} we get that $T^{\dagger}$ is bounded. Thus $$\|\frac{1}{(E(|u|^2))^{\frac{1}{2}}{(E(|w|^2))^{\frac{1}{2}}}}\|_{\infty}=\|T^{\dagger}\|\geq \gamma(T^{\dagger}).$$
By Proposition 2.12 of \cite{kn} we have
 $$\gamma(T)=\frac{1}{\|T^{\dagger}\|}=\|(E(|u|^2))^{\frac{1}{2}}{(E(|w|^2))^{\frac{1}{2}}}\|=\|T\|$$

Now by the above results, we have a necessary and sufficient condition for boundedness of the Moore-Penrose inverse of an unbounded WCT operator with closed range.
\begin{cor}
Under the assumptions of Theorem \ref{t2.6}, for $T=M_wEM_u$, we get that $T^{\dagger}$ is bounded if and only if $E(|u|^2)E(|w|^2)\geq r$, $\mu$, a.e., for some $r>0$. And in this case $\|T^{\dagger}\|=\|\frac{1}{(E(|u|^2))^{\frac{1}{2}}(E(|w|^2))^{\frac{1}{2}}}\|_{\infty}$.
\end{cor}
\begin{proof}
By the Proposition \ref{p2.8} we now that the $\mathcal{R}(T)$ is closed if and only if $T^{\dagger}$ is bounded. Since $$T^{\dagger}=M_{\frac{\chi_S}{E(|u|^2)E(|w|^2)}}M_{\bar{u}}EM_{\bar{w}}, \text{in which} \ \ S=S(E(|u|^2))\cap S(E(|w|^2)),$$
is a WCT operator, then by \cite{ej1} we get that $T^{\dagger}$ is bounded if and only if $$\frac{1}{(E(|u|^2))^{\frac{1}{2}}(E(|w|^2))^{\frac{1}{2}}}\in L^{\infty}(\mathcal{A}),$$
 and in this case $\|T^{\dagger}\|=\|\frac{1}{(E(|u|^2))^{\frac{1}{2}}(E(|w|^2))^{\frac{1}{2}}}\|_{\infty}$.
\end{proof}

Now we recall some properties of closed densely defined operators between two Hilbert spaces.
\begin{prop}
Let $\mathcal{H}_1$ and $\mathcal{H}_2$ be Hilbert spaces and $T:\mathcal{H}_1\rightarrow \mathcal{H}_2$ be a closed densely defined operator. Then the followings hold:
\begin{enumerate}
\item $\mathcal{N}(T)=\mathcal{R}(T^*)^{\perp}$.
\item $\mathcal{N}(T^*)=\mathcal{R}(T)^{\perp}$.
\item $\mathcal{N}(T^*T)=\mathcal{N}(T)$.
\item $\overline{\mathcal{R}(T^*T)}=\overline{\mathcal{R}(T^*)}$.
\end{enumerate}

\end{prop}

Let $T$ be a linear operator on the Hilbert space $\mathcal{H}$ and $\mathcal{G}(T)=\{(x,Tx): x\in \mathcal{D}(T)\}$ be the graph of $T$. Let $P=(P_{ij})$ be the projection from $\mathcal{H}\oplus \mathcal{H}$ onto $\overline{\mathcal{G}(T)}$. The $2\times 2$ operator $P$ in $\mathcal{H}$ is called the characteristic matrix of $T$. Since $P$ is a projection, then $P^2=P^*=P$ and so
$$P^*_{ij}=P_{ij}, \ \ \ \ \ \ \ \ \ \Sigma_{k}P_{ik}P_{kj}=P_{ij}.$$
Therefore $P_{11}$ and $P_{22}$ are selfadjoint. Moreover, $TP_{11}=P_{21}$ and $TP_{12}=P_{22}$. Now we recall the next Proposition for characteristic matrix for regular operators.
\begin{prop}\cite{ft}\label{p2.13}
Let $T$ be a regular operator with characteristic matrix $P=(P_{ij})$. Then we have the followings:
\begin{itemize}
  \item The entries $P_{ij}$ are given by
  $$P_{11}=(I+T^*T)^{-1} \ \ \ \ \ \ \ \ P_{12}=T^*(I+TT^*)^{-1},$$
  $$P_{21}=T(I+T^*T)^{-1} \ \ \ \ \ \ \ \ P_{22}=TT^*(I+TT^*)^{-1}.$$
  \item $I-P_{22}=(I+TT^*)^{-1}$.
\end{itemize}

\end{prop}
Let $T$ be a closed densely defined operator on the Hilbert space $\mathcal{H}$ and $\lambda>0$. Then the operator $\lambda I+ T^*T$ is strictly positive and bounded below. Indeed, for $\lambda>0$ and $f\in \mathcal{D}(T^*T)$, we have

$$\langle (\lambda I+ T^*T)f,f\rangle=\langle \lambda^{\frac{1}{2}}f,\lambda^{\frac{1}{2}}f\rangle+\langle Tf,Tf\rangle\geq \langle \lambda^{\frac{1}{2}}f,\lambda^{\frac{1}{2}}f\rangle=\lambda\|f\|^2>0.$$
This implies that the operator $\lambda I+T^*T$ is strictly positive, bounded below, one to one and so invertible.\\

Here we recall some basic calculations in functional analysis to obtain some properties of the positive operator $P_{\lambda}=(\lambda I+T^*T)^{-1}$. Let

$$U,V:\mathcal{H}\times \mathcal{H}\rightarrow \mathcal{H}\times \mathcal{H}, \text{such that}, \ \ \ \ V(f,g)=(g,-f), \ \ \ U(f,g)=(g,f).$$

Then we have
$$\mathcal{G}_{T}\oplus V\mathcal{G}_{T^*}=\mathcal{G}_{T^*}\oplus V\mathcal{G}_{T}=\mathcal{H}\times \mathcal{H}.$$

Let $h\in \mathcal{H}$. Then $(h,0)\in \mathcal{H}\times \mathcal{H}$ and so $(h,0)=(x,Tx)+(T^*y,-y)$ for $x\in \mathcal{D}(T)$ and $y\in \mathcal{D}(T^*)$. Therefore $h=x+T^*y$ and $y=Tx$. So we have $h=x+T^*Tx=(I+T^*T)x$ and consequently $x=(I+T^*T)^{-1}h=P_{11}h$ and $y=Tx=TP_{11}h$. This means that the operators $P_{11}$ and $TP_{11}$ are every where defined. Therefore, we get that $P_{\lambda}$ is also every where defined.\\

In addition for $x\in \mathcal{H}$,
\begin{align*}
\langle P_{\lambda}x,x\rangle&=\langle P_{\lambda}x,(\lambda I+T^*T)P_{\lambda}x\rangle\\
&=\lambda\langle P_{\lambda}x,P_{\lambda}x\rangle+\langle P_{\lambda}x,T^*TP_{\lambda}x\rangle\\
&=\lambda\langle P_{\lambda}x,P_{\lambda}x\rangle+\langle TP_{\lambda}x,TP_{\lambda}x\rangle\\
&\geq0.
\end{align*}
Then $P_{\lambda}$ is positive. Moreover, for $x,y\in \mathcal{H}$,
\begin{align*}
\langle P_{\lambda} x,y\rangle&=\langle P_{\lambda}x,P_{\lambda}^{-1}P_{\lambda}y\rangle\\
&=\langle P_{\lambda}x,(\lambda I+T^*T)T_{\lambda}y\rangle\\
&=\lambda\langle P_{\lambda}x,P_{\lambda}y\rangle+\langle P_{\lambda}x,T^*Ty\rangle\\
&=\lambda\langle P_{\lambda}x,P_{\lambda}y\rangle+\langle T^*TP_{\lambda}x,y\rangle\\
&=\langle (\lambda I+T^*T)T_{\lambda}x,P_{\lambda}y\rangle\\
&=\langle x,P_{\lambda}y\rangle.
\end{align*}
This implies that $P_{\lambda}$ is symmetric. In the next theorem apply the above observations to WCT operators.

\begin{thm}\label{t2.14}
Let $T=M_wEM_u$ be densely defined on $L^2(\mu)$. Then the characteristic matrix of $T$ is:

\begin{itemize}
\item $P_{11}=I-M_{h_1\bar{u}}EM_u$,\\
\item $P_{21}=M_hM_wEM_u$,\\
\item $P_{12}=M_hM_{\bar{u}}EM_{\bar{w}}$,\\
\item $P_{22}=M_{h_2}M_{w}EM_{\bar{w}}$,\\
\end{itemize}

in which

$$h_1=\frac{E(|w|^2)}{1+E(|u|^2)E(|w|^2)},  \ \ \ \ h_2=\frac{E(|u|^2)}{1+E(|u|^2)E(|w|^2)}, \ \ \ \ h=\frac{1}{1+E(|u|^2)E(|w|^2)}.$$

\end{thm}
\begin{proof}

In \cite{y1} we have found that for densely defined WCT operator $T=M_wEM_u$ on $L^2(\mu)$ and $\lambda \notin \text{ess.range}(E(uw))$,

$$(\lambda I-M_wEM_u)^{-1}=\frac{1}{\lambda} I+M_{\frac{w}{\lambda(\lambda-E(uw))}}EM_u.$$

Also, we have proven that if $E(| w|^2)E(| u|^2)<\infty$, a.e., $\mu$, then $T=M_wEM_u$ and $T^*$ are closed and densely defined on the Hilbert space $L^2(\mu)$, and also $T^{**}=T$. This means that $T$ is regular.\\

 By the fact that $T^*T=M_{E(|w|^2)\bar{u}}EM_{u}$ and $TT^*=M_{E(|u|^2)w}EM_{\bar{w}}$ are positive operators and the above formula from \cite{y1}, and also by Proposition \ref{p2.13} we have the characteristic matrix of WCT operator $T=M_wEM_u$ as follows:

\begin{align*}
P_{11}&=(I+M_{\bar{u}E(|w|^2)}E_u)^{-1}=I-M_{h_1\bar{u}}EM_u,\\
P_{21}&=M_{w}EM_{u}(I+M_{\bar{u}E(|w|^2)}EM_u)^{-1}=M_hM_wEM_u,\\
P_{12}&=M_{\bar{u}}EM_{\bar{w}}(I+M_{wE(|u|^2)}EM_{\bar{w}})^{-1}=M_hM_{\bar{u}}EM_{\bar{w}},\\
P_{22}&=M_{wE(|u|^2)}EM_{\bar{w}}(I+M_{wE(|u|^2)}EM_{\bar{w}})^{-1}=M_{h_2}M_{w}EM_{\bar{w}},\\
\end{align*}

in which

$$h_1=\frac{E(|w|^2)}{1+E(|u|^2)E(|w|^2)},  \ \ \ \ h_2=\frac{E(|u|^2)}{1+E(|u|^2)E(|w|^2)}, \ \ \ \ h=\frac{1}{1+E(|u|^2)E(|w|^2)}.$$

\end{proof}

In the sequel we get some similar results, indeed for every $\lambda>0$ we have

  $$P_{\lambda}=(\lambda I+M_{\bar{u}E(| w|^2)}EM_u)^{-1}=\frac{1}{\lambda}I-M_{\frac{E(| w|^2)\bar{u}}{\lambda(E(| w|^2)E(| u|^2)+\lambda)}}EM_u$$

  $$Q_{\lambda}=(\lambda I+M_{wE(| u|^2)}EM_{\bar{w}})^{-1}=\frac{1}{\lambda}I-M_{\frac{E(| u|^2)w}{\lambda(E(| w|^2)E(| u|^2)+\lambda)}}EM_{\bar{w}}$$

$$TP_{\lambda}=\frac{1}{\lambda}M_wEM_u-M_{\frac{E(| w|^2)E(|u|^2)w}{\lambda(E(|w|^2)E(| u|^2))+\lambda}}EM_u=M_{t_{\lambda}}T$$
i.e.,
$$TP_{\lambda}=Q_{\lambda}T=M_{t_{\lambda}}T=TM_{t_{\lambda}},$$

in which $t_{\lambda}=\frac{1}{E(|u|^2)E(|w|^2)+\lambda}$.\\

By Theorem \ref{t2.9} we have
$$(I-P_{11})^{\dagger}=M_{\chi_{S}\frac{\bar{u}}{(E(|u|^2))^2}}EM_u$$
and
$$(I-Q_1)^{\dagger}=M_{\frac{\chi_{S}w}{E(|w|^2)}}EM_{\bar{w}}.$$

As we compute in \cite{ej1},
$$|T|=M_{g\bar{u}}EM_u, \ \ \ \ \text{in which} \ \ \ \ g=(\frac{E(|w|^2)}{E(|u|^2)})^{\frac{1}{2}}\chi_S,$$
and so by Remark \ref{r2.5} we have
$$\sigma(|T|)\backslash \{0\}=\text{ess.range}((E(|w|^2))^{\frac{1}{2}}(E(|u|^2))^{\frac{1}{2}})\backslash \{0\}.$$
By the above observations we get the next theorem.
\begin{thm}
 Let $T=M_wEM_u:L^2(\mu)\rightarrow L^2(\mu)$ and $(E(|u|^2))^{\frac{1}{2}}(E(|w|^2))^{\frac{1}{2}}<\infty, \ \ \ \text{a.e.}$. For every $\lambda>0$, the operators $P_{\lambda}$ and $TP_{\lambda}$ are every where defined, bounded and contraction. Moreover, $P_{\lambda}$ and $Q_{\lambda}$ are positive and symmetric.
\end{thm}

\section{\sc\bf Variational Regularization}

In this section we assume that $Tu=f$ is a solvable linear equation in a Hilbert space $\mathcal{H}$, $T$ is a linear, closed, densely defined operator in $\mathcal{H}$, which is not boundedly invertible. In \cite{ram} the authors obtained some results for unbounded densely defined operators on the Hilbert space $\mathcal{H}$ and then they applied the results to the variational regularization problem $F(u):=\|Tu-f\|^2+\lambda \|u\|^2=min$, in which $f\in \mathcal{H}$. In this section we determine and prove that the unique norm minimal-norm solution of the equation $Tu=f$ is $T^{\dagger}f$. First we discuss it for the WCT operators on the Hilbert space $L^2(\mu)$.\\

As we have obtained, if $T=M_wEM_u$ is densely defined, then $M_{t_{\lambda}}T^*$ is a densely defined, closed WCT operator. Now we have a nice relation between Moore-Penrose inverse of the WCT operator $T$ and a minimal norm solution of the equation $Tu=f$.
\begin{thm}\label{t3.1}
Let $T=M_wEM_u$ be a densely defined WCT operator on the Hilbert space $L^2(\mu)$. For every $f\in L^2(\mu)$, the linear functional $F(u)=\|Tu-f\|^2+\lambda\|u\|^2$ has a unique global minimizer $u_{\lambda}=M_{t_{\lambda}}T^{*}(f)$, where $\lambda>0$ and
$P_{\lambda}T^*(f)=Q_{\lambda}T^*(f)=M_{t_{\lambda}}T^*(f)$,
in which $t_{\lambda}=\frac{1}{\lambda+E(|u|^2)E(|w|^2)}$. If $f\in \mathcal{R}(T)$, then
$$\lim_{\lambda\rightarrow 0}\|u_{\lambda}-y\|=0,$$
where $u_{\lambda}$ is the unique global minimizer of $F(u)$ and $y$ is the minimal-norm solution to the equation $Tu=f$. In this case the minimal-norm solution is exactly the value of the Moore-Penrose inverse of $T$ at $f$, i.e., $T^{\dagger}(f)$.
\end{thm}
\begin{proof}
By using Theorem 2 from \cite{ram} and Theorem \ref{t2.9} we get the proof.
\end{proof}

\begin{exam}
Let $X=[-1,1]$, $d\mu=\frac{1}{2}dx$ and $\mathcal{A}$ be the sigma algebra generated by symmetric intervals in $[0,1]$i.e., $\mathcal{A}=<\{(-a,a):0\leq a\leq1\}>$.

Then we have
 $$E^{\mathcal{A}}(f)(t)=\frac{f(t)+f(-t)}{2}, \ \ t\in X,$$
 where $E^{\mathcal{A}}(f)$ is defined. If $u(t)=e^{t}$ and $w(t)=\sin(t)+\cos(t)$, then
 $$E^{\mathcal{A}}(\mid u\mid^2)(t)=\cosh(2t),\ \ \ \ \ \ \ \ \ E^{\mathcal{A}}(\mid w\mid^2)(t)=1.$$
Also,
$$Tf(t)=M_wEM_u(f)(t)=(\sin(t)+\cos(t))(\frac{e^tf(t)+e^{-1}f(-t)}{2}).$$

It's clear that $T=M_wEM_u$ is densely defined. Since
$$E^{\mathcal{A}}(\mid u\mid^2)(t)E^{\mathcal{A}}(\mid w\mid^2)(t)=E^{\mathcal{A}}(\mid u\mid^2)(t),$$
then it's bounded away from zero on $[-1,1]$. Hence by Theorem \ref{t2.6} we get that $T$ has closed range. So $T^{\dagger}$ exists and

$$T^{\dagger}=M_{\frac{e^t}{\cosh(2t)}}EM_{\sin(t)+\cos(t)}.$$
Hence by Theorem \ref{t3.1} we get that the minimal norm solution of the equation:
$$(\sin(t)+\cos(t))(e^tf(t)+e^{-t}f(-t))=2g$$
is
$$f(t)=\frac{(\sin(t)+\cos(t))e^tg(t)+(\cos(t)-\sin(t))e^tg(-t)}{2\cosh(2t)}.$$

 Also, by Theorem \ref{t2.14} we have the characteristic matrix of $T$ as follow:

$$P=
\left(
  \begin{array}{cc}
    I-M_{\frac{e^{t}}{1+\cosh(2t)}}EM_{e^{t}} & M_{\frac{e^{t}}{1+\cosh(2t)}}EM_{\sin(t)+\cos(t)} \\
    M_{\frac{\sin(t)+\cos(t)}{1+\cosh(2t)}}EM_{e^{t}} & M_{\frac{\cosh(2t)(\sin(t)+\cos(t))}{1+\cosh(2\alpha t)}}EM_{\sin(t)+\cos(t)} \\
  \end{array}
\right).
 $$
  \end{exam}

\begin{exam} Let $X=[0,1]\times
[0,1]$, $d\mu=dtdt'$, $\Sigma$  the  Lebesgue subsets of $\Omega$ and
let $\mathcal{A}=\{A\times [0,1]: A \ \mbox{is a Lebesgue set in}
\ [0,1]\}$. Then, for each $f$ in $L^2(\Sigma)$, $(Ef)(t,
t')=\int_0^1f(t,s)ds$, which is independent of the second
coordinate. Hence for every finite valued measurable functions $u,w$ on $X$, the linear operator
$$T(f)(t,t')=w(t,t')\int^{1}_{0}u(t,s)f(t,s)ds$$
is densely defined on $L^2(X)$. Also, by Theorem \ref{t3.1} we get that the solution of each equation of the form:
$$w(t,t')\int^1_0u(t,s)f(t,s)ds=g(t,t')$$
is
$$f(t,t')=\frac{u(t,t')}{\int^{1}_{0}\mid w\mid^2(t,s)ds+\int^{1}_{0}\mid u\mid^2(t,s)ds}\int^1_0w(t,s)g(t,s)ds.$$

 This example is valid in the general case as follows:\\
Let $(X_1,\Sigma_1, \mu_1)$ and $(X_2,\Sigma_2, \mu_2)$ be two
$\sigma$-finite measure spaces and $X=X_1\times X_2$,
$\Sigma=\Sigma_1\times \Sigma_2$ and $\mu=\mu_1\times \mu_2$. Put
$\mathcal{A}=\{A\times X_2:A\in \Sigma_1\}$. Then $\mathcal{A}$ is
a sub-$\sigma$-algebra of $\Sigma$. Then for all $f$ in domain
$E^{\mathcal{A}}$ we have
$$E^{\mathcal{A}}(f)(t_1)=E^{\mathcal{A}}(f)(t_1,t_2)=\int_{X_2}f(t_1,s)d\mu_2(s) \ \ \
\mu-a.e.$$ on $X$.\\

Also, if $(X,\Sigma, \mu)$ is a finite measure space and
$u:X\times X\rightarrow \mathbb{C}$ and $w:X\rightarrow \mathbb{C}$ are finite valued measurable functions,
then the operator
$T:L^2(\Sigma)\rightarrow L^2(\Sigma)$ defined by
$$Tf(t)=w(t)\int_{X}u(t,s)f(s)d\mu(s), \ \ \ \ \ f\in L^2(\Sigma),$$
is called weighted kernel operator on $L^2(\Sigma)$). We show that $T$ is a
WCT operator.\cite{gd} For each $f\in L^2(\Sigma)$, let $f'(t_1,t_2)=f(t_2)$, for $t_1,t_2\in X$. If we set
$\mathcal{A}=\{A\times X:A\in \Sigma\}$. Then $\mathcal{A}$ is
a sub-$\sigma$-algebra of $\Sigma\times \Sigma$ and for every $f\in L^2(\Sigma)$, $f'$ is $\Sigma\times \Sigma$-measurable. Then for all $f$ in domain
$E^{\mathcal{A}}$ we have
$$E^{\mathcal{A}}(f')(t_1)=E^{\mathcal{A}}(f')(t_1,t_2)=\int_{X}f'(t_1,s)d\mu(s)= \int_{X}f(s)d\mu(s)\ \ \
\mu-a.e.$$ on $X$.\\

 It is clear that $u(t_1,t_2)f'(t_1,t_2)=u(t_1,t_2)f(t_2)$ is a $\Sigma\otimes \Sigma$-measurable function, when $f\in
L^2(\Sigma)$. Therefore we have
\begin{align*}
w(t)E^{\mathcal{A}}(uf)(t)&=E^{\mathcal{A}}(uf')(t,s)\\
&=w(t)\int_{X}u(t,t')f'(,t')d\mu(t')\\
&=w(t)\int_{X}u(t,t')f(t')d\mu(t')\\
&=Tf(t).\\
\end{align*}
Hence $T=M_wEM_u$, i.e, $T$ is a WCT operator. And since $w$ and $u$ are finite valued, then the WCT operator $T$ is densely defined.
Therefore, by Theorem \ref{t3.1} we can find the solution of Fredholm integral equations, i.e., the solution of the integral equation of the form

$$w(t)\int^{b}_{a}u(t,t')f(t')d\mu(t')=g(t), \ \ \ \ \ \ \ \ a,b,t\in \mathbb{R}$$
is
$$f(t)=\frac{u(t,t')}{\int^{b}_{a}\mid w\mid^2(s)ds+\int^{b}_{a}\mid u\mid^2(t,s)ds}\int^b_aw(s)g(s)ds.$$

This means all assertions of this paper are valid for a class of integral type operators.
\end{exam}

We recall that if $T$ is a linear densely defined closed operator on the Hilbert space $\mathcal{H}$, then the domain of $T^*$ is also dense in $\mathcal{H}$ and so $T^{**}=(T^*)^*$ exists and $T^{**}=T$. So in the next proposition we find a limit for a unique global minimizer of a linear functional
\begin{prop}\label{p3.2}
Let $\mathcal{H}$ be a Hilbert space $T$ be a linear, unbounded and densely defined operator in $\mathcal{H}$. For any $x\in \mathcal{H}$ if we set $R_{\lambda}(x):=u_{\lambda}=T^*Q_{\lambda}(x)$, then
$$\lim_{\lambda\rightarrow 0}R_{\lambda}(x)=\lim_{\lambda\rightarrow 0}u_{\lambda}=T^{\dagger}(x).$$
\end{prop}
\begin{proof}
It is easy to see that the operator $T^*Q_{\lambda}$ is a bounded operator. And direct computations show that
$$TR_{\lambda}=I-\lambda Q_{\lambda}$$
and so
$$R^*_{\lambda}T^*=I-\lambda Q_{\lambda}.$$
Hence we get that
$$\lim_{\lambda\rightarrow 0} R^*_{\lambda}T^*=I.$$
This implies that
$$T^*\lim_{\lambda\rightarrow 0} R^*_{\lambda}T^*=T^*$$
and therefore
$$TST=T,$$
in which $S=\lim_{\lambda\rightarrow 0}R_{\lambda}$. Similarly we get that
$$(ST)^*=ST$$
and also we have $ST=P_{\overline{\mathcal{R}(S)}}$ and $TS=P_{\overline{\mathcal{R}(T)}}$.
All these properties confirm that $S=T^{\dagger}$ and so $\lim_{\lambda\rightarrow 0}u_{\lambda}=T^{\dagger}(x)$.

\end{proof}
Now in the net theorem under the assumptions of the Proposition \ref{p3.2} we get that the value of the Moore-Penrose inverse of an unbounded operator $T$ at the point $f$ is a minimal norm solution of the equation $Tu=f$.
\begin{thm}
Let $f\in \mathcal{H}$ and $F(u)=\|Tu-f\|^2+\lambda\|u\|^2$. Then the linear functional $F$ has a unique global minimizer $u_{\lambda}=T^*Q_{\lambda}$, where $\lambda>0$, $Q_{\lambda}=(TT^*+\lambda I)^{-1}$, $P_{\lambda}=(T^*T+\lambda I)^{-1}$ and
$$T^*Q_{\lambda}f=P_{\lambda}T^*f,$$
where $P_{\lambda}$ is the closure of $P_{\lambda}T^*$ that is defined on $\mathcal{D}(T^*)$. If $f\in \mathcal{R}(T)$, then
$$\lim_{\lambda\rightarrow 0}\|u_{\lambda}-T^{\dagger}f\|=0,$$
i.e., $T^{\dagger}f$ is the minimal norm solution to the equation $Tu=f$.
\end{thm}
\begin{proof}It is a direct consequence of Theorem 2 of \cite{ram}.
\end{proof}

\end{document}